\documentclass{article}

\usepackage[english]{babel}\usepackage[latin1]{inputenc}\usepackage[T1]{fontenc}
\usepackage{mathdots,graphicx,amsfonts,amssymb,amsmath,amsthm,mathrsfs,epstopdf,enumitem,multirow,booktabs,subcaption,color,mathtools}
\usepackage[top=2cm,bottom=3cm,left=2cm,right=2cm]{geometry}
\usepackage{sectsty}\sectionfont{\normalsize}\subsectionfont{\normalsize}

\newtheorem{theorem}{Theorem}[section]
\newtheorem{lemma}{Lemma}[section]
\theoremstyle{definition}
\newtheorem{definition}{Definition}[section]
\newtheorem{remark}{Remark}[section]
\newtheorem{example}{Example}[section]

\numberwithin{equation}{section}
\numberwithin{figure}{section}
\numberwithin{table}{section}

\renewcommand{\epsilon}{\varepsilon}
\newcommand{\vv}{{\boldsymbol v}}
\newcommand{\xx}{{\boldsymbol x}}

\begin{document}

\title{Block Jacobi/Gauss--Seidel preconditioning for GLT~sequences, and~GLH~sequences}
\author{Carlo Garoni\\
\footnotesize Department of Mathematics, University of Rome Tor Vergata, Rome, Italy (garoni@mat.uniroma2.it)\\[10pt]
Stefano Serra-Capizzano\thanks{Corresponding author}\\
\footnotesize Department of Science and High Technology, University of Insubria, Como, Italy (s.serracapizzano@uninsubria.it)\\[-2pt]
\footnotesize Department of Information Technology, Uppsala University, Uppsala, Sweden (stefano.serra@it.uu.se)}
\date{}

\maketitle

\begin{abstract}
The theory of generalized locally Toeplitz (GLT) sequences is an apparatus for computing the spectral and singular value distribution of sequences of matrices that possess a (possibly hidden) Toeplitz-like structure. These sequences, which are known as GLT sequences, arise in several applications, including the discretization of differential equations.
Associated with any GLT sequence is a special function called symbol. In this paper, we prove that, if $\{A_n\}_n$ is a GLT sequence with symbol $\kappa$ and $P_n$ is any block Jacobi or block Gauss--Seidel preconditioner for $A_n$ with a fixed number of blocks independent of $n$, then $\{P_n\}_n$ is a GLT sequence with symbol $\kappa$, just like $\{A_n\}_n$.
This result allows us to predict a remarkable efficiency of block Jacobi/Gauss--Seidel preconditioning for GLT sequences, which is in fact illustrated through numerical experiments.
It also allows us to extend the Fasino--Tilli theorem on the zero distribution of Hankel matrix sequences generated by $L^1$ functions to a larger class of matrix sequences called generalized locally Hankel (GLH) sequences.

\smallskip


\smallskip

\end{abstract}


\section{Introduction}


Throughout this paper, a sequence of matrices is, by definition, a sequence of the form $\{A_n\}_n$, where $n$ varies in some infinite subset of $\mathbb N$ and $A_n$ is a matrix of size $d_n\times e_n$ such that both $d_n$ and $e_n$ tend to $\infty$ as $n\to\infty$.
The theory of generalized locally Toeplitz (GLT) sequences is an apparatus for computing the spectral and singular value distribution of sequences of matrices that possess a (possibly hidden) Toeplitz-like structure. Sequences of this kind, which are known as GLT sequences, arise in several applications, including the discretization of differential and integral equations.
Nowadays, the theory of GLT sequences is a fairly extensive research area with numerous applications.
For readers who are new to the subject, we recommend the introduction \cite{GLT-intro} and the six-page conference paper \cite{aip}.
For a comprehensive exposition of the topic, we refer to the books \cite{GLTbookI,GLTbookII} and the book-like papers \cite{rg,bgR,bg,bgd}.
Recent noteworthy developments not published in the previous references include the identification between spaces of GLT sequences and function spaces \cite{b2017,krs2022}, the characterization of diagonal matrix sequences enjoying a spectral distribution in terms of GLT sequences \cite{b2019}, the spectral distribution result for non-Hermitian perturbations of Hermitian GLT sequences \cite{b2020}, the GLT classification of diagonal sampling matrix sequences obtained from quasi-uniform samples of almost everywhere continuous functions \cite{griglie_a.u.}, the derivation of a ``normal form'' for GLT sequences \cite{normal_form}, and the successful application of matrix-less spectral approximation methods to GLT sequences \cite{bit2025}.
It is also worth noting that, as demonstrated by recent research, the spectral distribution of a GLT sequence has significant practical implications. For example, suppose that $\{A_n\}_n$ is a GLT sequence resulting from the discretization of a differential equation $\mathscr Au=f$ through a given numerical method. Then, the spectral distribution of $\{A_n\}_n$ can be used to measure the accuracy of the method in approximating the spectrum of the differential operator $\mathscr A$ \cite{DavideCalcolo2021}, to establish whether the method preserves the so-called average spectral gap \cite{DavideCalcolo2018}, or to formulate analytical predictions for the eigenvalues of both $A_n$ and $\mathscr A$ \cite{Tom-paper}. Moreover, the spectral distribution of $\{A_n\}_n$ can be exploited to design efficient iterative solvers for linear systems with matrix $A_n$ and to analyze/predict their performance; see \cite{BeKu,Kuij-SIREV} for accurate convergence estimates of Krylov methods based on the spectral distribution and \cite[p.~3]{GLTbookI} for more details on this subject.

Associated with any GLT sequence $\{A_n\}_n$ is a special function $\kappa$, the so-called symbol.
Given a GLT sequence $\{A_n\}_n$ with symbol $\kappa$ and a vector $(n_1,\ldots,n_\nu)$ of fixed length $\nu$ such that $n_1+\ldots+n_\nu={\rm size}(A_n)$, we consider two sequences of preconditioners for $A_n$:
\begin{itemize}[nolistsep,leftmargin=*]
	\item the sequence of block Jacobi preconditioners for $A_n$ associated with $(n_1,\ldots,n_\nu)$, i.e., the sequence of matrices
	\[ P_n^{J}=\mathop{\rm blockdiag}_{n_1,\ldots,n_\nu}(A_n)=\begin{bmatrix}A_n^{(11)} & & \\ & \ddots & \\ & & A_n^{(\nu\nu)}\end{bmatrix}, \]
	where $A_n^{(kk)}$ is the $k$th diagonal block of $A_n$ of size $n_k$;
	\item the sequence of block Gauss--Seidel preconditioners for $A_n$ associated with $(n_1,\ldots,n_\nu)$, i.e., the sequence of matrices
	\[ P_n^{GS}=\mathop{\rm blocktril}_{n_1,\ldots,n_\nu}(A_n)=\begin{bmatrix}A_n^{(11)} & & \\ \vdots & \ddots & \\ A_n^{(\nu1)} & \cdots & A_n^{(\nu\nu)}\end{bmatrix}, \]
	where $A_n^{(pq)}$ is the block of $A_n$ of size $n_p\times n_q$ in position $(p,q)$.
\end{itemize}
	In our main results (Theorems~\ref{blockdiag(GLT)=GLT}--\ref{blocktril(GLT)=GLT}), we prove that, regardless of $(n_1,\ldots,n_\nu)$, $\{P_n^J\}_n$ and $\{P_n^{GS}\}_n$ are GLT sequences with symbol $\kappa$, just like $\{A_n\}_n$. This result has two noteworthy implications.
\begin{itemize}[nolistsep,leftmargin=*]
	\item First, it gives reasons to believe that block Jacobi/Gauss--Seidel preconditioning for matrices $A_n$ belonging to a GLT sequence is very efficient; see Remark~\ref{bJp-efficiency}.
	The numerical experiments presented in this paper reveal that this is in fact the case.
	\item Secondly, it allows us to give an alternative proof of the Fasino--Tilli theorem on the zero distribution of sequences of Hankel matrices generated by $L^1$ functions (Theorem~\ref{thm:FT}). Actually, it also allows us to extend the Fasino--Tilli theorem to a larger class of matrix sequences that we call generalized locally Hankel (GLH) sequences; see Theorem~\ref{thm:FTextended}.
\end{itemize}

The paper is organized as follows.
In Section~\ref{sec:overview}, we give an overview of the theory of GLT sequences with a focus on the results that we need in this paper.
In Section~\ref{sec:GLTp}, we extend to the case of GLT sequences with matrix-valued symbols a well-known preconditioning result, which will be used in combination with our first main result (Theorem~\ref{blockdiag(GLT)=GLT}) to theoretically justify the efficiency of block Jacobi preconditioners for GLT sequences.
In Section~\ref{sec:bJGSp-GLT}, we state and prove our main results described above (Theorems~\ref{blockdiag(GLT)=GLT}--\ref{blocktril(GLT)=GLT}).
In Section~\ref{sec:GLH}, as an application of Theorem~\ref{blockdiag(GLT)=GLT}, we give an alternative proof of the Fasino--Tilli theorem on the zero distribution of sequences of Hankel matrices generated by $L^1$ functions, and we also extend this theorem to GLH sequences.
In Section~\ref{sec:ne}, we illustrate through numerical examples the efficiency of block Jacobi and block Gauss--Seidel preconditioners for GLT sequences.
In Section~\ref{sec:conc}, we draw conclusions and suggest future lines of research, which include, among others, the application of our main results to the analysis of additive and multiplicative Schwarz methods in the context of domain decomposition techniques.

\section{Overview of the theory of GLT sequences}\label{sec:overview}

In this section, we give an overview of the theory of GLT sequences.
For conciseness purposes, we only present the results that we need in this paper.
For a comprehensive exposition of the topic, see \cite{rg,bgR,bg,bgd,GLTbookI,GLTbookII}.
For an introduction to the subject, we recommend \cite{GLT-intro,aip}.

\bigskip

\noindent\textbf{Singular value and spectral distribution of a sequence of matrices.}
Let $\mu_k$ be the Lebesgue measure in $\mathbb R^k$. Throughout this paper, all terminology from measure theory (such as ``measurable set'', ``measurable function'', ``a.e.'', etc.)\ always refers to the Lebesgue measure. A matrix-valued function $f:D\subseteq\mathbb R^k\to\mathbb C^{s\times t}$ is said to be measurable (respectively, continuous, continuous a.e., in $L^1(D)$, etc.)\ if its components $f_{ij}:D\to\mathbb C$, $i=1,\ldots,s$, $j=1,\ldots,t$, are measurable (respectively, continuous, continuous a.e., in $L^1(D)$, etc.).
For every $x,y\in\mathbb R$, we define $x\wedge y=\min(x,y)$.
We denote by $C_c(\mathbb C)$ (respectively, $C_c(\mathbb R)$) the space of complex-valued functions defined on $\mathbb C$ (respectively, $\mathbb R$) with bounded support.
The singular values of a matrix $A\in\mathbb C^{m\times n}$ are denoted by $\sigma_i(A)$, $i=1,\ldots,m\wedge n$, and the eigenvalues of a matrix $A\in\mathbb C^{m\times m}$ are denoted by $\lambda_i(A)$, $i=1,\ldots,m$. The spectrum of a matrix $A\in\mathbb C^{m\times m}$ is denoted by $\Lambda(A)$. 
Recall that a sequence of matrices is, by definition, a sequence of the form $\{A_n\}_n$, where $n$ varies in some infinite subset of $\mathbb N$ and $A_n$ is a matrix of size $d_n\times e_n$ such that both $d_n$ and $e_n$ tend to $\infty$ as $n\to\infty$.

\begin{definition}[\textbf{singular value and spectral distribution of a sequence of matrices}]\label{def-distribution}\hfill
\begin{itemize}[leftmargin=*,nolistsep]
	\item Let $\{A_n\}_n$ be a sequence of matrices with $A_n$ of size $d_n\times e_n$, and let $f:D\subset\mathbb R^k\to\mathbb C^{s\times t}$ be measurable with $0<\mu_k(D)<\infty$. We say that $\{A_n\}_n$ has a singular value distribution described by $f$, and we write $\{A_n\}_n\sim_\sigma f$, if
	\begin{equation*}
	\lim_{n\rightarrow\infty}\frac1{d_n\wedge e_n}\sum_{i=1}^{d_n\wedge e_n}F(\sigma_i(A_n))=\frac1{\mu_k(D)}\int_D\frac{\sum_{i=1}^{s\wedge t}F(\sigma_i(f(\xx)))}{s\wedge t}{{\rm d}}\xx,\qquad\forall\,F\in C_c(\mathbb R).
	\end{equation*}
	\item Let $\{A_n\}_n$ be a sequence of square matrices with $A_n$ of size $d_n$, and let $f:D\subset\mathbb R^k\to\mathbb C^{s\times s}$ be measurable with $0<\mu_k(D)<\infty$. We say that $\{A_n\}_n$ has a spectral (or eigenvalue) distribution described by $f$, and we write $\{A_n\}_n\sim_\lambda f$, if
	\begin{equation*}
	\lim_{n\rightarrow\infty}\frac1{d_n}\sum_{i=1}^{d_n}F(\lambda_i(A_n))=\frac1{\mu_k(D)}\int_D\frac{\sum_{i=1}^{s}F(\lambda_i(f(\xx)))}{s}{{\rm d}}\xx,\qquad\forall\,F\in C_c(\mathbb C).
	\end{equation*}
\end{itemize}
\end{definition}

Note that Definition~\ref{def-distribution} is well-posed by \cite[Lemma~2.1]{bg}, which ensures that the functions $\xx\mapsto\sum_{i=1}^{s\wedge t}F(\sigma_i(f(\xx)))$ and $\xx\mapsto\sum_{i=1}^{s}F(\lambda_i(f(\xx)))$ are measurable. 
We refer the reader to \cite[Remark~2.9]{bg} or \cite[Remarks~4.1--4.2]{blocking} for the informal meaning behind the singular value and spectral distribution of a sequence of matrices.

The following lemma will be used in Section~\ref{sec:GLTp}; see \cite[Lemma~2.3 and Corollary~2.15]{bg} for the corresponding proof.
Throughout this paper, the closure of a set $S$ is denoted by $\overline S$.

\begin{lemma}\label{S1-block}
Let $\{A_n\}_n$ be a sequence of square matrices and let $f:D\subset\mathbb R^k\to\mathbb C^{s\times s}$ be measurable with $0<\mu_k(D)<\infty$. If $\{A_n\}_n\sim_\lambda f$ and $\Lambda(A_n)$ is contained in $S\subseteq\mathbb C$ for all $n$, then $\Lambda(f)\subseteq\overline S$ a.e.
\end{lemma}

\smallskip

\noindent\textbf{Zero-distributed sequences.}
A zero-distributed sequence is a sequence of matrices $\{Z_n\}_n$ such that $\{Z_n\}_n\sim_\sigma0$, i.e.,
\[ \lim_{n\to\infty}\frac1{d_n\wedge e_n}\sum_{i=1}^{d_n\wedge e_n}F(\sigma_i(Z_n))=F(0),\qquad\forall\,F\in C_c(\mathbb R), \]
where $d_n\times e_n$ is the size of $Z_n$.

\bigskip

\noindent\textbf{Toeplitz sequences.}
If $f:[-\pi,\pi]\to\mathbb C^{s\times t}$ is in $L^1([-\pi,\pi])$, its Fourier coefficients are denoted by $\{f_k\}_{k\in\mathbb Z}$ and are defined as follows:
\begin{equation}\label{Fc}
f_k=\frac1{2\pi}\int_{-\pi}^{\pi}f(\theta){\rm e}^{-{\rm i}k\theta}{\rm d}\theta\:\in\:\mathbb{C}^{s\times t},\qquad k\in\mathbb Z,
\end{equation}
where the integrals are computed componentwise. The $n$th (block) Toeplitz matrix generated by a function $f:[-\pi,\pi]\to\mathbb C^{s\times t}$ belonging to $L^1([-\pi,\pi])$ is the $ns\times nt$ matrix denoted by $T_n(f)$ and defined as follows:
\[ T_n(f)=[f_{i-j}]_{i,j=1}^n. \]
Any sequence of matrices of the form $\{T_{d_n}(f)\}_n$, with $d_n\to\infty$ and $f:[-\pi,\pi]\to\mathbb C^{s\times t}$ in $L^1([-\pi,\pi])$, is referred to as a (block) Toeplitz sequence generated by $f$.

\bigskip

\noindent\textbf{Sequences of diagonal sampling matrices.}
If $a:[0,1]\to\mathbb C^{s\times t}$, the $n$th (block) diagonal sampling matrix generated by $a$ is the $ns\times nt$ (block) diagonal matrix denoted by $D_n(a)$ and defined as follows:
\[ D_n(a)=\mathop{\rm diag}_{i=1,\ldots,n}a\Bigl(\frac in\Bigr). \]
Any sequence of matrices of the form $\{D_{d_n}(a)\}_n$, with $d_n\to\infty$ and $a:[0,1]\to\mathbb C^{s\times t}$, is referred to as a sequence of (block) diagonal sampling matrices generated by $a$.

\bigskip

\noindent\textbf{Approximating classes of sequences.}
The notion of approximating classes of sequences (a.c.s.)\ is the cornerstone of an asymptotic approximation theory for sequences of matrices that has been developed since the last years of the XX century; see \cite[Chapter~5]{GLTbookI}. The formal definition of a.c.s.\ is reported in Definition~\ref{a.c.s.}. Throughout this paper, we denote by $\|\cdot\|$ the Euclidean norm ($2$-norm) of vectors and the associated operator norm (induced norm) for matrices, which is defined by $\|A\|=\max\{\|A\xx\|:\xx\in\mathbb C^n,\,\|\xx\|=1\}$ for every $A\in\mathbb C^{m\times n}$.

\begin{definition}[\textbf{approximating class of sequences}]\label{a.c.s.}
Let $\{A_n\}_n$ be a sequence of matrices with $A_n$ of size $d_n\times e_n$, and let $\{\{B_{n,m}\}_n\}_m$ be a sequence of sequences of matrices with $B_{n,m}$ of size $d_n\times e_n$. We say that $\{\{B_{n,m}\}_n\}_m$ is an approximating class of sequences (a.c.s.)\ for $\{A_n\}_n$, and we write $\{B_{n,m}\}_n\xrightarrow{\rm a.c.s.}\{A_n\}_n$, if the following condition is met: for every $m$ there exists $n_m$ such that, for $n\ge n_m$,
\begin{equation*}
A_n=B_{n,m}+R_{n,m}+N_{n,m},\qquad{\rm rank}(R_{n,m})\le c(m)(d_n\wedge e_n),\qquad\|N_{n,m}\|\le\omega(m),
\end{equation*}
where $n_m,\,c(m),\,\omega(m)$ depend only on $m$ and $\displaystyle\lim_{m\to\infty}c(m)=\lim_{m\to\infty}\omega(m)=0$.
\end{definition}

\smallskip

\noindent\textbf{GLT sequences.}
In this section, we collect the properties of GLT sequences that we need in this paper.
We begin with the following definition, which is an equivalent alternative to the usual (complicated) definition of GLT sequences.
Throughout this paper, we denote by $I_s$ the $s\times s$ identity matrix.

\begin{definition}[\textbf{generalized locally Toeplitz sequence}]\label{GLT_def}
Let $\{X_n\}_n$ be a sequence of matrices, with $X_n$ of size $d_ns\times d_nt$ for some fixed positive integers $s,t$ and some positive integer sequence $\{d_n\}_n$ tending to $\infty$, and let $\kappa:[0,1]\times[-\pi,\pi]\to\mathbb C^{s\times t}$ be measurable. We say that $\{X_n\}_n$ is a (block) generalized locally Toeplitz (GLT) sequence with symbol $\kappa$, and we write $\{X_n\}_n\sim_{\rm GLT}\kappa$, if there exist functions $a_{i,m}$, $f_{i,m}$, $i=1,\ldots,N_m$, such that
\begin{itemize}[nolistsep,leftmargin=*]
	\item $a_{i,m}:[0,1]\to\mathbb C$ is continuous a.e.\ on $[0,1]$ and $f_{i,m}:[-\pi,\pi]\to\mathbb C^{s\times t}$ belongs to $L^1([-\pi,\pi])$, \vspace{3pt}
	\item $\kappa_m(x,\theta)=\sum_{i=1}^{N_m}a_{i,m}(x)f_{i,m}(\theta)\to\kappa(x,\theta)$ a.e.\ on $[0,1]\times[-\pi,\pi]$, \vspace{3pt}
	\item $\{X_{n,m}\}_n=\bigl\{\sum_{i=1}^{N_m}D_{d_n}(a_{i,m}I_s)T_{d_n}(f_{i,m})\bigr\}_n\xrightarrow{\rm a.c.s.}\{X_n\}_n$.
\end{itemize}
\end{definition}

The properties of GLT sequences that we need in this paper are listed below; for the corresponding proofs, see \cite{bgR,bg}.
Before being able to formulate these properties, we need to introduce some notation and terminology.
Throughout this paper, we denote by $O_{s,t}$ the $s\times t$ zero matrix. The matrix $O_{s,s}$ is denoted by $O_s$. If the size is clear from the context, we often write $O$ instead of $O_{s,t}$ and $O_s$. The conjugate transpose of a matrix $A$ is denoted by $A^*$. If $A$ is a matrix, we denote by $A^\dag$ the Moore--Penrose pseudoinverse of $A$. What is relevant for our purposes is that $A^\dag=A^{-1}$ whenever $A$ is invertible. For more details on the pseudoinverse of a matrix, see \cite{Bini,GV}. If $A$ is a diagonalizable matrix and $f$ is a complex-valued function defined at each point of $\Lambda(A)$, we denote by $f(A)$ the unique matrix such that $f(A)\vv=f(\lambda)\vv$ whenever $A\vv=\lambda\vv$. Note that a spectral decomposition of $A$ immediately yields a spectral decomposition of $f(A)$:
\begin{equation}\label{A->f(A)}
A=V\begin{bmatrix}\lambda_1 & & \\ & \ddots & \\ & & \lambda_n\end{bmatrix}V^{-1}\quad\implies\quad f(A)=V\begin{bmatrix}f(\lambda_1) & & \\ & \ddots & \\ & & f(\lambda_n)\end{bmatrix}V^{-1}.
\end{equation}
For more on matrix functions, we refer the reader to Higham's book~\cite{Higham}.
Hereafter, the composite function $f\circ g$ is denoted by $f(g)$.

\begin{enumerate}[nolistsep,leftmargin=37pt]
	\item[\textbf{GLT0.}] Let $\{X_n\}_n$ be a sequence of matrices, with $X_n$ of size $d_ns\times d_nt$ for some fixed positive integers $s,t$ and some positive integer sequence $\{d_n\}_n$ tending to $\infty$, and let $\kappa,\xi:[0,1]\times[-\pi,\pi]\to\mathbb C^{s\times t}$ be measurable.
	\begin{itemize}[nolistsep,leftmargin=*]
		\item If $\{X_n\}_n\sim_{\rm GLT}\kappa$ and $\kappa=\xi$ a.e., then $\{X_n\}_n\sim_{\rm GLT}\xi$.
		\item If $\{X_n\}_n\sim_{\rm GLT}\kappa$ and $\{X_n\}_n\sim_{\rm GLT}\xi$, then $\kappa=\xi$ a.e.
	\end{itemize}
	\item[\textbf{GLT1.}] Let $\{X_n\}_n$ be a sequence of matrices, with $X_n$ of size $d_ns\times d_nt$ for some fixed positive integers $s,t$ and some positive integer sequence $\{d_n\}_n$ tending to $\infty$, and let $\kappa:[0,1]\times[-\pi,\pi]\to\mathbb C^{s\times t}$ be measurable.
	\begin{itemize}[nolistsep,leftmargin=*]
		\item If $\{X_n\}_n\sim_{\rm GLT}\kappa$, then $\{X_n\}_n\sim_\sigma\kappa$.
		\item If $\{X_n\}_n\sim_{\rm GLT}\kappa$ and the matrices $X_n$ are Hermitian, then $s=t$, $\kappa$ is Hermitian a.e., and $\{X_n\}_n\sim_\lambda\kappa$.
	\end{itemize}
	\item[\textbf{GLT2.}] Let $s,t$ be positive integers and let $\{d_n\}_n$ be a positive integer sequence tending to $\infty$. Then,
	\begin{itemize}[nolistsep,leftmargin=*]
		\item $\{T_{d_n}(f)\}_n\sim_{\rm GLT}\kappa(x,\theta)=f(\theta)$ if $f:[-\pi,\pi]\to\mathbb C^{s\times t}$ belongs to $L^1([-\pi,\pi])$;
		\item $\{D_{d_n}(a)\}_n\sim_{\rm GLT}\kappa(x,\theta)=a(x)$ if $a:[0,1]\to\mathbb C^{s\times t}$ is continuous a.e.;
		\item for every sequence of matrices $\{Z_n\}_n$ with $Z_n$ of size $d_ns\times d_nt$, we have
		$\{Z_n\}_n\sim_{\rm GLT}\kappa(x,\theta)=O_{s,t}$ if and only if $\{Z_n\}_n\sim_\sigma0$.
	\end{itemize}
	\item[\textbf{GLT3.}] Let $\{X_n\}_n$, $\{Y_n\}_n$ be sequences of matrices, with $X_n$ of size $d_ns\times d_nt$ and $Y_n$ of size $d_nu\times d_nv$ for some fixed positive integers $s,t,u,v$ and some positive integer sequence $\{d_n\}_n$ tending to $\infty$, and let $\kappa:[0,1]\times[-\pi,\pi]\to\mathbb C^{s\times t}$ and $\xi:[0,1]\times[-\pi,\pi]\to\mathbb C^{u\times v}$ be measurable. Suppose that $\{X_n\}_n\sim_{\rm GLT}\kappa$ and $\{Y_n\}_n\sim_{\rm GLT}\xi$. Then,
	\begin{itemize}[nolistsep,leftmargin=*]
		\item $\{X_n^*\}_n\sim_{\rm GLT}\kappa^*$;
		\item $\{\alpha X_n+\beta Y_n\}_n\sim_{\rm GLT}\alpha\kappa+\beta\xi$ for every $\alpha,\beta\in\mathbb C$ if $\kappa$ and $\xi$ are summable (i.e., $s=u$ and $t=v$);
		\item $\{X_nY_n\}_n\sim_{\rm GLT}\kappa\xi$ if $\kappa$ and $\xi$ are multipliable (i.e., $t=u$);
		\item $\{X_n^\dag\}_n\sim_{\rm GLT}\kappa^{-1}$ if $\kappa$ is invertible a.e.;
		\item $\{f(X_n)\}_n\sim_{\rm GLT}f(\kappa)$ if the matrices $X_n$ are Hermitian and $f:\mathbb C\to\mathbb C$ is continuous.
	\end{itemize}
\end{enumerate}

\section{GLT preconditioning}\label{sec:GLTp}

The purpose if this section is to prove Theorem~\ref{exe-preconditioning}, which is used in Section~\ref{sec:bJGSp-GLT} to theoretically justify the efficiency of block Jacobi preconditioners for GLT sequences; see Remark~\ref{bJp-efficiency}.
Theorem~\ref{exe-preconditioning} is the generalization of \cite[Theorem~3.2]{GLT-intro} to the case where the involved GLT sequences have matrix-valued symbols.
To prove it, we need the following lemma.
Throughout this paper, we use the abbreviations HPD and HPSD for Hermitian positive definite and Hermitian positive semi-definite, respectively.

\begin{lemma}\label{GLT_HPSD->symbol_HPSD}
Let $\{A_n\}_n$ be a sequence of HPSD matrices, with $A_n$ of size $d_ns$ for some fixed positive integer $s$ and some positive integer sequence $\{d_n\}_n$ tending to $\infty$, and let $\kappa:[0,1]\times[-\pi,\pi]\to\mathbb C^{s\times s}$ be measurable. If $\{A_n\}_n\sim_{\rm GLT}\kappa$, then $\kappa$ is HPSD a.e.
\end{lemma}
\begin{proof}
Since the matrices $A_n$ are Hermitian, we know from {\bf GLT1} that $\kappa$ is Hermitian a.e.\ and $\{A_n\}_n\sim_\lambda\kappa$.
Since the matrices $A_n$ are positive semi-definite, $\Lambda(A_n)\subset[0,\infty)$ for all $n$.
Hence, by Lemma~\ref{S1-block}, $\Lambda(\kappa)\subset[0,\infty)$ a.e.
Thus, $\kappa$ is HPSD a.e.
\end{proof}

\begin{theorem}\label{exe-preconditioning}
Let $\{A_n\}_n$ be a sequence of Hermitian matrices, with $A_n$ of size $d_ns$ for some fixed positive integer $s$ and some positive integer sequence $\{d_n\}_n$ tending to $\infty$, and let $\{P_n\}_n$ be a sequence of HPD matrices, with $P_n$ of size $d_ns$.
Suppose that $\{A_n\}_n\sim_{\rm GLT}\kappa$ and $\{P_n\}_n\sim_{\rm GLT}\xi$, where $\kappa,\xi:[0,1]\times[-\pi,\pi]\to\mathbb C^{s\times s}$ are measurable and $\xi$ is invertible a.e.
Then, the sequence of preconditioned matrices $P_n^{-1}A_n$ satisfies
\begin{equation}\label{GLTp-GLTsigla}
\{P_n^{-1}A_n\}_n\sim_{\rm GLT}\xi^{-1}\kappa,\qquad\{P_n^{-1}A_n\}_n\sim_\sigma\xi^{-1}\kappa,\qquad\{P_n^{-1}A_n\}_n\sim_\lambda\xi^{-1}\kappa.
\end{equation}
\end{theorem}
\begin{proof}
The GLT relation in \eqref{GLTp-GLTsigla} is a direct consequence of {\bf GLT3}. The singular value distribution in \eqref{GLTp-GLTsigla} follows immediately from the GLT relation in \eqref{GLTp-GLTsigla} and {\bf GLT1}. We prove the spectral distribution in \eqref{GLTp-GLTsigla}.
Since $P_n$ is HPD, the eigenvalues of $P_n$ are positive and the matrices $P_n^{1/2}$ and $P_n^{-1/2}$ are well-defined as functions of $P_n$.
Note that $P_n^{1/2}P_n^{-1/2}=I_n$ and both $P_n^{1/2}$ and $P_n^{-1/2}$ are HPD.\,\footnote{\,These properties follow from \eqref{A->f(A)} applied with $A=P_n$ and with a unitary matrix $V$, which is possible because $P_n$ is HPD.}
Note also that
\begin{equation}\label{sim.sim}
P_n^{-1}A_n\sim P_n^{1/2}(P_n^{-1}A_n)P_n^{-1/2}=P_n^{-1/2}A_nP_n^{-1/2},
\end{equation}
where $X\sim Y$ means that $X$ is similar to $Y$.
The good news is that $P_n^{-1/2}A_nP_n^{-1/2}$ is Hermitian, because $P_n^{-1/2}$ and $A_n$ are both Hermitian.
By {\bf GLT3}---the last item of which is applied with $f(\lambda)=|\lambda|^{1/2}$---we have
\begin{align}
\{P_n^{1/2}\}_n&\sim_{\rm GLT}|\xi|^{1/2},\notag\\
\{P_n^{-1/2}\}_n=\{(P_n^{1/2})^{-1}\}_n&\sim_{\rm GLT}(|\xi|^{1/2})^{-1}=|\xi|^{-1/2},\notag\\
\{P_n^{-1/2}A_nP_n^{-1/2}\}_n&\sim_{\rm GLT}|\xi|^{-1/2}\kappa|\xi|^{-1/2}.\label{illabel}
\end{align}
Note that $\xi$ is HPD a.e.\ by Lemma~\ref{GLT_HPSD->symbol_HPSD} and the assumption that $\xi$ is invertible a.e.
Hence, $|\xi|^{-1/2}=\xi^{-1/2}$ a.e., and we infer from \eqref{illabel} and {\bf GLT0} that
\[ \{P_n^{-1/2}A_nP_n^{-1/2}\}_n\sim_{\rm GLT}\xi^{-1/2}\kappa\hspace{1pt}\xi^{-1/2}. \]
Since $P_n^{-1/2}A_nP_n^{-1/2}$ is Hermitian, {\bf GLT1} yields
\[ \{P_n^{-1/2}A_nP_n^{-1/2}\}_n\sim_\lambda\xi^{-1/2}\kappa\hspace{1pt}\xi^{-1/2}. \]
Thus, by the similarity \eqref{sim.sim}, $\{P_n^{-1}A_n\}_n\sim_\lambda\xi^{-1/2}\kappa\hspace{1pt}\xi^{-1/2}$. To conclude, we observe that
\[ \xi^{-1/2}\kappa\hspace{1pt}\xi^{-1/2}=\xi^{1/2}(\xi^{-1}\kappa)\xi^{-1/2}\sim\xi^{-1}\kappa\,\mbox{ a.e.} \]
This implies that $\{P_n^{-1}A_n\}_n\sim_\lambda\xi^{-1}\kappa$ by Definition~\ref{def-distribution}.
\end{proof}

\section{Block Jacobi and block Gauss--Seidel preconditioners for GLT sequences}\label{sec:bJGSp-GLT}

This section is devoted to the proof of our main results (Theorems~\ref{blockdiag(GLT)=GLT}--\ref{blocktril(GLT)=GLT}).
Throughout this paper, if $X\in\mathbb C^{m_1\times m_2}$ and $Y\in\mathbb C^{\ell_1\times\ell_2}$, the tensor (Kronecker) product of $X$ and $Y$ is the $m_1\ell_1\times m_2\ell_2$ matrix defined by
\[ X\otimes Y=[X_{ij}Y]_{i=1,\ldots,m_1}^{j=1,\ldots,m_2}=\begin{bmatrix}X_{11}Y & \cdots & X_{1m_2}Y\\ \vdots & & \vdots\\ X_{m_11}Y & \cdots & X_{m_1m_2}Y\end{bmatrix}, \]
and the direct sum of $X$ and $Y$ is the $(m_1+\ell_1)\times(m_2+\ell_2)$ matrix defined by
\[ X\oplus Y={\rm diag}(X,Y)=\begin{bmatrix}X & O\\ O & Y\end{bmatrix}. \]
A partition of a positive integer $n$ is a vector of positive integers $(n_1,\ldots,n_\nu)$ such that $n_1+\ldots+n_\nu=n$.

\begin{definition}[\textbf{block Jacobi and block Gauss--Seidel preconditioners}]
Let $A=[A_{ij}]_{i,j=1}^n\in\mathbb C^{ns\times nt}$ be a block matrix with blocks $A_{ij}\in\mathbb C^{s\times t}$ for every $i,j=1,\ldots,n$, and let $(n_1,\ldots,n_\nu)$ be a partition of $n$.
\begin{itemize}[nolistsep,leftmargin=*]
	\item The block Jacobi preconditioner of $A$ generated by the partition $(n_1,\ldots,n_\nu)$ is the matrix
	\[ \mathop{\rm blockdiag}_{n_1,\ldots,n_\nu}^{s,t}(A)=\begin{bmatrix}A^{(11)} & & \\ & \ddots & \\ & & A^{(\nu\nu)}\end{bmatrix}=A^{(11)}\oplus\cdots\oplus A^{(\nu\nu)}, \]
	where $A^{(kk)}$ is the $k$th diagonal block of $A$ of size $n_ks\times n_kt$, i.e., $A^{(kk)}=[A_{ij}]_{i,j=n_1+\ldots+n_{k-1}+1}^{n_1+\ldots+n_k}$.
	\item The block Gauss--Seidel preconditioner of $A$ generated by the partition $(n_1,\ldots,n_\nu)$ is the matrix
	\[ \mathop{\rm blocktril}_{n_1,\ldots,n_\nu}^{s,t}(A)=\begin{bmatrix}A^{(11)} & & \\ \vdots & \ddots & \\[3pt] A^{(\nu1)} & \cdots & A^{(\nu\nu)}\end{bmatrix}, \]
	where $A^{(pq)}$ is the block of $A$ of size $n_ps\times n_qt$ in position $(p,q)$, i.e., $A^{(pq)}=[A_{ij}]_{i=n_1+\ldots+n_{p-1}+1,\ldots,n_1+\ldots+n_p}^{j=n_1+\ldots+n_{q-1}+1,\ldots,n_1+\ldots+n_q}$.
\end{itemize}
\end{definition}

We are now ready to state and prove Theorems~\ref{blockdiag(GLT)=GLT}--\ref{blocktril(GLT)=GLT}. They show that any sequence of block Jacobi preconditioners or block Gauss--Seidel preconditioners for a GLT sequence with symbol $\kappa$ is again a GLT sequence with symbol $\kappa$, as long as the number of blocks in each preconditioner is fixed.
Throughout this paper, the characteristic (indicator) function of a set $E$ is denoted by $\chi_E$.

\begin{theorem}\label{blockdiag(GLT)=GLT}
Let $\nu,s,t\ge1$ be fixed integers, let $\{d_n\}_n$ be a positive integer sequence tending to $\infty$, and, for every~$n$, let $(n_1,\ldots,n_\nu)=(n_1(n),\ldots,n_\nu(n))$ be a partition of $d_n$.
Let $\kappa:[0,1]\times[-\pi,\pi]\to\mathbb C^{s\times t}$ be measurable and let $\{A_n\}_n$ be a sequence of matrices with $A_n$ of size $d_ns\times d_nt$ and $\{A_n\}_n\sim_{\rm GLT}\kappa$. Then,
\[ \Bigl\{\mathop{\rm blockdiag}_{n_1,\ldots,n_\nu}^{s,t}(A_n)\Bigr\}_n\sim_{\rm GLT}\kappa. \]
\end{theorem}
\begin{proof}
Let
\[ P_n^J=\mathop{\rm blockdiag}_{n_1,\ldots,n_\nu}^{s,t}(A_n)=\begin{bmatrix}A_n^{(11)} & & \\ & \ddots & \\ & & A_n^{(\nu\nu)}\end{bmatrix}. \]
We have to show that $\{P_n^J\}_n\sim_{\rm GLT}\kappa$. The proof is divided into two cases.

\medskip

\noindent
{\em Case 1.} First, we prove that $\{P_n^J\}_n\sim_{\rm GLT}\kappa$ under the additional assumption that there exists
\[ \lim_{n\to\infty}\frac{n_i}{\sum_{j=1}^\nu n_j}=c_i,\qquad i=1,\ldots,\nu. \]
For this proof, we draw inspiration from \cite[proof of Theorem~3.12]{blocking}.
Define $\rho_i=\sum_{j=1}^ic_j$ for $i=0,\ldots,\nu$ $(\rho_0=0)$ and note that 
\[ P_n^J=\sum_{i=1}^\nu P_{n,i}, \]
where $P_{n,i}$ is the matrix obtained from $P_n^J$ by setting to zero all blocks $A_n^{(kk)}$ except for $A_n^{(ii)}$.
We show that, for every $i=1,\ldots,\nu$,
\begin{equation}\label{tcp}
P_{n,i}=D_{d_n}(\psi_iI_s)A_nD_{d_n}(\psi_iI_t)+R_{n,i},
\end{equation}
where $\psi_k(x)=\chi_{(\rho_{k-1},\rho_k)}(x)$ for $k=1,\ldots,\nu$ and
\begin{equation}\label{r-tcp}
\lim_{n\to\infty}\frac{{\rm rank}(R_{n,i})}{d_n}=0.
\end{equation}
Once this is proved, in view of the equation
\[ P_n^J=\sum_{i=1}^\nu P_{n,i}=\sum_{i=1}^\nu D_{d_n}(\psi_iI_s)A_nD_{d_n}(\psi_iI_t)+\sum_{i=1}^\nu R_{n,i}, \]
and taking into account that each sequence $\{R_{n,i}\}_n$ is zero-distributed by \eqref{r-tcp}, 
the application of {\bf GLT2}--{\bf GLT3} yields
\[ \{P_n^J\}_n\sim_{\rm GLT}\tilde\kappa(x,\theta)=\sum_{i=1}^\nu\psi_i(x)\psi_i(x)\kappa(x,\theta)=\sum_{i=1}^\nu\chi_{(\rho_{i-1},\rho_i)}(x)\kappa(x,\theta), \]
where the latter equality holds because $\psi_i(x)\psi_i(x)=\chi_{(\rho_{i-1},\rho_i)}(x)$ for every $i=1,\ldots,\nu$.
Since the function $\tilde\kappa(x,\theta)=\sum_{i=1}^\nu\chi_{(\rho_{i-1},\rho_i)}(x)\kappa(x,\theta)$ is equal to $\kappa(x,\theta)$ a.e.\ in $[0,1]\times[-\pi,\pi]$, the GLT relation $\{P_n^J\}_n\sim_{\rm GLT}\tilde\kappa(x,\theta)$ is equivalent to the thesis $\{P_n^J\}_n\sim_{\rm GLT}\kappa(x,\theta)$ by {\bf GLT0}.

To conclude the proof, it only remains to prove \eqref{tcp}--\eqref{r-tcp}.
Set $\epsilon_k=c_k-n_k/d_n$ for $k=1,\ldots,\nu$ (note that $\epsilon_k\to0$ as $n\to\infty$ by assumption). Define $N_p=\sum_{k=1}^pn_k$ and $q_p=d_n\sum_{k=1}^p\epsilon_k$ for $p=0,\ldots,\nu$ ($N_0=q_0=0$). For $i=1,\ldots,\nu$, a direct computation shows that, for every $u,v=1,\ldots,d_n$, the $s\times t$ block in position $(u,v)$ of the $d_ns\times d_nt$ matrix $P_{n,i}$ is given by
\begin{align*}
(P_{n,i})_{uv}&=\left\{\begin{aligned}
&(A_n)_{uv}, &&\mbox{if $N_{i-1}<u,v\le N_i$,}\\
&O_{s,t}, &&\mbox{otherwise,}
\end{aligned}\right.
\end{align*}
and the $s\times t$ block in position $(u,v)$ of the $d_ns\times d_nt$ matrix $D_{d_n}(\psi_iI_s)A_nD_{d_n}(\psi_iI_t)$ is given by
\begin{align*}
(D_{d_n}(\psi_iI_s)A_nD_{d_n}(\psi_iI_t))_{uv}&=(D_{d_n}(\psi_iI_s))_{uu}(A_n)_{uv}(D_{d_n}(\psi_iI_t))_{vv}=\psi_i\Bigl(\frac u{d_n}\Bigr)I_s\,(A_n)_{uv}\,\psi_i\Bigl(\frac v{d_n}\Bigr)I_t\\
&=\left\{\begin{aligned}
&(A_n)_{uv}, &&\mbox{if $N_{i-1}+q_{i-1}<u,v<N_i+q_i$,}\\
&O_{s,t}, &&\mbox{otherwise.}
\end{aligned}\right.
\end{align*}
Therefore, setting $R_{n,i}=P_{n,i}-D_{d_n}(\psi_iI_s)A_nD_{d_n}(\psi_iI_t)$, for every $u,v=1,\ldots,d_n$, the $s\times t$ block $(R_{n,i})_{uv}$ is certainly equal to zero in the following cases:
\begin{itemize}[nolistsep,leftmargin=*]
	\item both $u$ and $v$ belong to
	\[ \bigl(\max(N_{i-1},N_{i-1}+q_{i-1}),\:\min(N_i,N_i+q_i)\bigr); \]
	\item at least one between $u$ and $v$ does not belong to 
	\[ \bigl[\min(N_{i-1},N_{i-1}+q_{i-1}),\:\max(N_i,N_i+q_i)\bigr]. \]
\end{itemize}
The number of integers $u$ belonging to $[\min(N_{i-1},N_{i-1}+q_{i-1}),\max(N_{i-1},N_{i-1}+q_{i-1})]$ is bounded by $|q_{i-1}|+1$ and, similarly, the number of integers $u$ belonging to $[\min(N_i,N_i+q_i),\max(N_i,N_i+q_i)]$ is bounded by $|q_i|+1$. 
It follows that
\[ {\rm rank}(R_{n,i})\le (s+t)(|q_{i-1}|+|q_i|+2)\le(s+t)\biggl(2d_n\sum_{k=1}^\nu|\epsilon_k|+2\biggr)=o(d_n), \]
and \eqref{tcp}--\eqref{r-tcp} are proved.

\medskip

\noindent
{\em Case 2.}
Now, we prove that $\{P_n^J\}_n\sim_{\rm GLT}\kappa$ without imposing additional assumptions.
Let $Z_n=P_n^J-A_n$. By {\bf GLT2}--{\bf GLT3}, the thesis $\{P_n^J\}_n\sim_{\rm GLT}\kappa$ is equivalent to $\{Z_n\}_n\sim_\sigma0$. Suppose by contradiction that the thesis does not hold. Then, $\{Z_n\}_n$ is not a zero-distributed sequence. This means that there exist $F\in C_c(\mathbb R)$ and a subsequence of indices $n\in\mathcal I$ such that
\begin{equation}\label{lim}
\lim_{\substack{\vphantom{\int}n\to\infty\\n\in\mathcal I}}\frac1{d_ns\wedge d_nt}\sum_{i=1}^{d_ns\wedge d_nt}F(\sigma_i(Z_n))=L\ne0.
\end{equation}
We extract from $\mathcal I$ a subsequence of indices $\mathcal I_1\subseteq\mathcal I$ such that there exists
\[ \lim_{\substack{\vphantom{\int}n\to\infty\\n\in\mathcal I_1}}\frac{n_1}{\sum_{j=1}^\nu n_j}=c_1; \]
then, we extract from $\mathcal I_1$ a subsequence of indices $n\in\mathcal I_2\subseteq\mathcal I_1$ such that there exists
\[ \lim_{\substack{\vphantom{\int}n\to\infty\\n\in\mathcal I_2}}\frac{n_2}{\sum_{j=1}^\nu n_j}=c_2; \]
and so on. This extraction procedure ends up with a subsequence of indices $n\in\mathcal I_\nu\subseteq\ldots\subseteq\mathcal I_1\subseteq\mathcal I$ such that there exists
\[ \lim_{\substack{\vphantom{\int}n\to\infty\\n\in\mathcal I_\nu}}\frac{n_i}{\sum_{j=1}^\nu n_j}=c_i,\qquad i=1,\ldots,\nu. \]
By the result of Case~1, we have $\{P_n^J\}_{n\in\mathcal I_\nu}\sim_{\rm GLT}\kappa$, which is equivalent to $\{Z_n\}_{n\in\mathcal I_\nu}\sim_\sigma0$ by {\bf GLT2}--{\bf GLT3} and the obvious GLT relation $\{A_n\}_{n\in\mathcal I_\nu}\sim_{\rm GLT}\kappa$.
In particular, we have
\[ \lim_{\substack{\vphantom{\int}n\to\infty\\n\in\mathcal I_\nu}}\frac1{d_ns\wedge d_nt}\sum_{i=1}^{d_ns\wedge d_nt}F(\sigma_i(Z_n))=0, \]
which is a contradiction to \eqref{lim}.
\end{proof}

\begin{theorem}\label{blocktril(GLT)=GLT}
Let $\nu,s,t\ge1$ be fixed integers, let $\{d_n\}_n$ be a positive integer sequence tending to $\infty$, and, for every~$n$, let $(n_1,\ldots,n_\nu)=(n_1(n),\ldots,n_\nu(n))$ be a partition of $d_n$.
Let $\kappa:[0,1]\times[-\pi,\pi]\to\mathbb C^{s\times t}$ be measurable and let $\{A_n\}_n$ be a sequence of matrices with $A_n$ of size $d_ns\times d_nt$ and $\{A_n\}_n\sim_{\rm GLT}\kappa$. Then,
\[ \Bigl\{\mathop{\rm blocktril}_{n_1,\ldots,n_\nu}^{s,t}(A_n)\Bigr\}_n\sim_{\rm GLT}\kappa. \]
\end{theorem}
\begin{proof}
The proof follows the same pattern as the proof of Theorem~\ref{blockdiag(GLT)=GLT}. However, we do not skip the details for the sake of completeness.
Let
\[ P_n^{GS}=\mathop{\rm blocktril}_{n_1,\ldots,n_\nu}^{s,t}(A_n)=\begin{bmatrix}A_n^{(11)} & & \\ \vdots & \ddots & \\ A_n^{(\nu1)} & \cdots & A_n^{(\nu\nu)}\end{bmatrix}. \]
We have to show that $\{P_n^{GS}\}_n\sim_{\rm GLT}\kappa$. The proof is divided into two cases.

\medskip

\noindent
{\em Case 1.} First, we prove that $\{P_n^{GS}\}_n\sim_{\rm GLT}\kappa$ under the additional assumption that there exists
\[ \lim_{n\to\infty}\frac{n_i}{\sum_{j=1}^\nu n_j}=c_i,\qquad i=1,\ldots,\nu. \]
Define $\rho_i=\sum_{j=1}^ic_j$ for $i=0,\ldots,\nu$ $(\rho_0=0)$ and note that
\[ P_n^{GS}=\sum_{1\le j\le i\le\nu}P_{n,i,j}, \]
where $P_{n,i,j}$ is the matrix obtained from $P_n^{GS}$ by setting to zero all blocks $A_n^{(pq)}$ except for $A_n^{(ij)}$.
We show that, for every $1\le j\le i\le\nu$,
\begin{equation}\label{tcp'}
P_{n,i,j}=D_{d_n}(\psi_iI_s)A_nD_{d_n}(\psi_jI_t)+R_{n,i,j},
\end{equation}
where $\psi_k(x)=\chi_{(\rho_{k-1},\rho_k)}(x)$ for $k=1,\ldots,\nu$ and
\begin{equation}\label{r-tcp'}
\lim_{n\to\infty}\frac{{\rm rank}(R_{n,i,j})}{d_n}=0.
\end{equation}
Once this is proved, in view of the equation
\[ P_n^{GS}=\sum_{1\le j\le i\le\nu}P_{n,i,j}=\sum_{1\le j\le i\le\nu}D_{d_n}(\psi_iI_s)A_nD_{d_n}(\psi_jI_t)+\sum_{1\le j\le i\le\nu}R_{n,i,j}, \]
and taking into account that each sequence $\{R_{n,i,j}\}_n$ is zero-distributed by \eqref{r-tcp'}, the application of {\bf GLT2}--{\bf GLT3} yields
\[ \{P_n^{GS}\}_n\sim_{\rm GLT}\tilde\kappa(x,\theta)=\sum_{1\le j\le i\le\nu}\psi_i(x)\psi_j(x)\kappa(x,\theta)=\sum_{i=1}^\nu\chi_{(\rho_{i-1},\rho_i)}(x)\kappa(x,\theta), \]
where the latter equality holds because $\psi_i(x)\psi_j(x)=0$ identically whenever $i\ne j$ and $\psi_i(x)\psi_i(x)=\chi_{(\rho_{i-1},\rho_i)}(x)$ for every $i=1,\ldots,\nu$.
Since the function $\tilde\kappa(x,\theta)=\sum_{i=1}^\nu\chi_{(\rho_{i-1},\rho_i)}(x)\kappa(x,\theta)$ is equal to $\kappa(x,\theta)$ a.e.\ in $[0,1]\times[-\pi,\pi]$, the GLT relation $\{P_n^{GS}\}_n\sim_{\rm GLT}\tilde\kappa(x,\theta)$ is equivalent to the thesis $\{P_n^{GS}\}_n\sim_{\rm GLT}\kappa(x,\theta)$ by {\bf GLT0}.

To conclude the proof, it only remains to prove \eqref{tcp'}--\eqref{r-tcp'}.
Set $\epsilon_k=c_k-n_k/d_n$ for $k=1,\ldots,\nu$ (note that $\epsilon_k\to0$ as $n\to\infty$ by assumption). Define $N_p=\sum_{k=1}^pn_k$ and $q_p=d_n\sum_{k=1}^p\epsilon_k$ for $p=0,\ldots,\nu$ ($N_0=q_0=0$). For $1\le j\le i\le\nu$, a direct computation shows that, for every $u,v=1,\ldots,d_n$, the $s\times t$ block in position $(u,v)$ of the $d_ns\times d_nt$ matrix $P_{n,i,j}$ is given by
\begin{align*}
(P_{n,i,j})_{uv}&=\left\{\begin{aligned}
&(A_n)_{uv}, &&\mbox{if $N_{i-1}<u\le N_i$ and $N_{j-1}<v\le N_j$,}\\
&O_{s,t}, &&\mbox{otherwise,}
\end{aligned}\right.
\end{align*}
and the $s\times t$ block in position $(u,v)$ of the $d_ns\times d_nt$ matrix $D_{d_n}(\psi_iI_s)A_nD_{d_n}(\psi_jI_t)$ is given by
\begin{align*}
(D_{d_n}(\psi_iI_s)A_nD_{d_n}(\psi_jI_t))_{uv}&=(D_{d_n}(\psi_iI_s))_{uu}(A_n)_{uv}(D_{d_n}(\psi_jI_t))_{vv}=\psi_i\Bigl(\frac u{d_n}\Bigr)I_s\,(A_n)_{uv}\,\psi_j\Bigl(\frac v{d_n}\Bigr)I_t\\
&=\left\{\begin{aligned}
&(A_n)_{uv}, &&\mbox{if $N_{i-1}+q_{i-1}<u<N_i+q_i$ and $N_{j-1}+q_{j-1}<v<N_j+q_j$,}\\
&O_{s,t}, &&\mbox{otherwise.}
\end{aligned}\right.
\end{align*}
Therefore, setting $R_{n,i,j}=P_{n,i,j}-D_{d_n}(\psi_iI_s)A_nD_{d_n}(\psi_jI_t)$, for every $u,v=1,\ldots,d_n$, the $s\times t$ block $(R_{n,i,j})_{uv}$ is certainly equal to zero in the following cases:
\begin{itemize}[nolistsep,leftmargin=*]
	\item $u$ belongs to
	\[ \bigl(\max(N_{i-1},N_{i-1}+q_{i-1}),\:\min(N_i,N_i+q_i)\bigr) \]
	and $v$ belongs to
	\[ \bigl(\max(N_{j-1},N_{j-1}+q_{j-1}),\:\min(N_j,N_j+q_j)\bigr); \]
	\item $u$ does not belong to 
	\[ \bigl[\min(N_{i-1},N_{i-1}+q_{i-1}),\:\max(N_i,N_i+q_i)\bigr] \]
	or $v$ does not belong to
	\[ \bigl[\min(N_{j-1},N_{j-1}+q_{j-1}),\:\max(N_j,N_j+q_j)\bigr]. \]
\end{itemize}
The number of integers $u$ belonging to $[\min(N_{i-1},N_{i-1}+q_{i-1}),\max(N_{i-1},N_{i-1}+q_{i-1})]$ is bounded by $|q_{i-1}|+1$ and, similarly, the number of integers $u$ belonging to $[\min(N_i,N_i+q_i),\max(N_i,N_i+q_i)]$ is bounded by $|q_i|+1$. Likewise, the number of integers $v$ belonging to $[\min(N_{j-1},N_{j-1}+q_{j-1}),\max(N_{j-1},N_{j-1}+q_{j-1})]$ is bounded by $|q_{j-1}|+1$ and, similarly, the number of integers $v$ belonging to $[\min(N_j,N_j+q_j),\max(N_j,N_j+q_j)]$ is bounded by $|q_j|+1$.
It follows that
\[ {\rm rank}(R_{n,i,j})\le s(|q_{i-1}|+|q_i|+2)+t(|q_{j-1}|+|q_j|+2)\le(s+t)\biggl(2d_n\sum_{k=1}^\nu|\epsilon_k|+2\biggr)=o(d_n), \]
and \eqref{tcp'}--\eqref{r-tcp'} are proved.

\medskip

\noindent
{\em Case 2.}
Now, we prove that $\{P_n^{GS}\}_n\sim_{\rm GLT}\kappa$ without imposing additional assumptions. Actually, there is nothing to prove, because the proof is this case is verbatim the same as the proof of Case~2 in Theorem~\ref{blockdiag(GLT)=GLT}, with the only difference that each occurrence of ``\,$P_n^J$\,'' must be replaced by ``\,$P_n^{GS}$\,''.
\end{proof}

\begin{remark}\label{bJp-efficiency}
Based on Theorems~\ref{blockdiag(GLT)=GLT}--\ref{blocktril(GLT)=GLT}, we can expect that block Jacobi/Gauss--Seidel preconditioning is efficient when applied to GLT sequences.
To see this, let us argue as follows. Let $\{A_n\}_n\sim_{\rm GLT}\kappa$, where $A_n$ is a square matrix of size $d_ns$ and $\kappa:[0,1]\times[-\pi,\pi]\to\mathbb C^{s\times s}$ is measurable, and let $(n_1,\ldots,n_\nu)=(n_1(n),\ldots,n_\nu(n))$ be a partition of $d_n$ with a fixed length $\nu$ independent of $n$.
By Theorems~\ref{blockdiag(GLT)=GLT}--\ref{blocktril(GLT)=GLT}, we have
\begin{equation}\label{pica1}
\Bigl\{P_n^J=\mathop{\rm blockdiag}_{n_1,\ldots,n_\nu}^{s,s}(A_n)\Bigr\}_n\sim_{\rm GLT}\kappa,\qquad\Bigl\{P_n^{GS}=\mathop{\rm blocktril}_{n_1,\ldots,n_\nu}^{s,s}(A_n)\Bigr\}_n\sim_{\rm GLT}\kappa.
\end{equation}
Assuming that the matrices $P_n^J$, $P_n^{GS}$ are invertible and $\kappa$ is invertible a.e., \eqref{pica1} and {\bf GLT3} yield
\begin{equation}\label{pica2}
\{(P_n^J)^{-1}A_n\}_n\sim_{\rm GLT}\kappa^{-1}\kappa=I_s,\qquad\{(P_n^{GS})^{-1}A_n\}_n\sim_{\rm GLT}\kappa^{-1}\kappa=I_s.
\end{equation}
The GLT relations \eqref{pica2} often imply the spectral distributions $\{(P_n^J)^{-1}A_n\}_n\sim_\lambda I_s$ and $\{(P_n^{GS})^{-1}A_n\}_n\sim_\lambda I_s$, which, by Definition~\ref{def-distribution}, are equivalent to
\begin{equation}\label{pica3}
\{(P_n^J)^{-1}A_n\}_n\sim_\lambda1,\qquad\{(P_n^{GS})^{-1}A_n\}_n\sim_\lambda1.
\end{equation}
For example, the spectral distribution $\{(P_n^J)^{-1}A_n\}_n\sim_\lambda1$ holds whenever the matrices $A_n$ are HPD and $\kappa$ is invertible a.e. In this case, the matrices $P_n^J$ are HPD and the spectral distribution $\{(P_n^J)^{-1}A_n\}_n\sim_\lambda1$ follows from 
Theorem~\ref{exe-preconditioning}.
What is important to point out is that the spectral distributions \eqref{pica3} are equivalent to the (weak) cluster at $1$ of the eigenvalues of $(P_n^J)^{-1}A_n$ and $(P_n^{GS})^{-1}A_n$; see \cite[Section~2.4.2]{bg} for details.
Thus, in view of the convergence properties of Krylov methods, we can expect that $P_n^J$ and $P_n^{GS}$ are efficient preconditioners for $A_n$.
Numerical experiments in support of this expectation are presented in Section~\ref{sec:ne}.
\end{remark}

\section{GLH sequences and the extended Fasino--Tilli theorem}\label{sec:GLH}

In this section, as an application of Theorem~\ref{blockdiag(GLT)=GLT}, we provide an alternative proof of the Fasino--Tilli theorem on the zero distribution of Hankel matrix sequences generated by $L^1$ functions; see Theorem~\ref{thm:FT}. This alternative proof naturally leads to both the definition of GLH sequences (Definition~\ref{GLH_def}) and the extension of the Fasino--Tilli theorem to such sequences; see Theorem~\ref{thm:FTextended}.
We first recall that the $n$th (block) Hankel matrix generated by a function $f:[-\pi,\pi]\to\mathbb C^{s\times t}$ belonging to $L^1([-\pi,\pi])$ is the $ns\times nt$ matrix denoted by $H_n(f)$ and defined as follows:
\[ H_n(f)=[f_{i+j-1}]_{i,j=1}^n=\left[\begin{array}{ccccc}f_1 & f_2 & f_3 & \cdots & f_n\\ f_2 & f_3 & & \iddots & f_{n+1}\\ f_3 & & \iddots & & f_{n+2}\\ \vdots & \iddots & & & \vdots\\ f_n & \cdots & \cdots & \cdots & f_{2n-1}\end{array}\right], \]
where $\{f_k\}_{k\in\mathbb Z}$ are the Fourier coefficients of $f$ defined in \eqref{Fc}.
Any sequence of matrices of the form $\{H_{d_n}(f)\}_n$, with $d_n\to\infty$ and $f:[-\pi,\pi]\to\mathbb C^{s\times t}$ in $L^1([-\pi,\pi])$, is referred to as a (block) Hankel sequence generated by $f$.
The following theorem is due to Fasino and Tilli \cite{FT}.

\begin{theorem}[\textbf{Fasino--Tilli theorem}]\label{thm:FT}
Let $f:[-\pi,\pi]\to\mathbb C^{s\times t}$ be in $L^1([-\pi,\pi])$. Then, $\{H_n(f)\}_n\sim_\sigma0$.
\end{theorem}

The proof of Theorem~\ref{thm:FT} that we are going to see is based on Theorem~\ref{blockdiag(GLT)=GLT} and is different from the original one appeared in~\cite{FT}.
In addition to Theorem~\ref{blockdiag(GLT)=GLT}, the proof requires Lemma~\ref{lemma_sigma} \cite[Theorem~3.10]{blocking}.
Throughout this paper, we denote by $W_n$ the $n\times n$ anti-identity matrix (or flip matrix) defined by
\begin{equation}\label{flipm}
W_n=\begin{bmatrix*}
&&&&1\\
&&&1\\
&&\iddots\\
&1\\
1
\end{bmatrix*}.
\end{equation}
A matrix $P\in\mathbb C^{m\times n}$ is said to be a compression matrix if $m\ge n$ and $P^*P=I_n$.

\begin{lemma}\label{lemma_sigma}
Let $\{X_n\}_n$ be a sequence of matrices with $X_n$ of size $d_n\times e_n$, and let $\{P_n\}_n$, $\{P_n'\}_n$ be sequences of compression matrices such that $P_n\in\mathbb C^{d_n\times\delta_n}$, $P_n'\in\mathbb C^{e_n\times\epsilon_n}$, and
\begin{equation*}
\liminf_{n\to\infty}\frac{\delta_n}{d_n}>0,\qquad\liminf_{n\to\infty}\frac{\epsilon_n}{e_n}>0.
\end{equation*}
If $\{X_n\}_n\sim_\sigma0$ then $\{P_n^*X_nP_n'\}_n\sim_\sigma0$.
\end{lemma}

\begin{proof}[Proof of Theorem~{\rm\ref{thm:FT}}]
Consider the matrix
\begin{align}
T_{2n}(f)&=[f_{i-j}]_{i,j=1}^{2n}=\begin{bmatrix}
f_0 & f_{-1} & f_{-2} & \cdots & \cdots & \cdots & f_{-(2n-1)}\\
f_1 & f_0 & f_{-1} & f_{-2} & & & \vdots \\
f_2 & f_1 & f_0 & f_{-1} & f_{-2} & & \vdots\\
\vdots & \ddots & \ddots & \ddots & \ddots & \ddots & \vdots\\
\vdots & & f_2 & f_1 & f_0 & f_{-1} & f_{-2}\\
\vdots & & & f_2 & f_1 & f_0 & f_{-1}\\[5pt]
f_{2n-1} & \cdots & \cdots & \cdots & f_2 & f_1 & f_0
\end{bmatrix}\notag\\
&=\left[\begin{array}{c|c}
[f_{i-j}]_{i,j=1}^n & [f_{i-j}]_{i=1,\ldots,n}^{j=n+1,\ldots,2n}\vphantom{\Big|}\\
\hline
[f_{i-j}]_{i=n+1,\ldots,2n}^{j=1,\ldots,n} & [f_{i-j}]_{i,j=n+1}^{2n}\vphantom{\Big|}
\end{array}\right]
=\left[\begin{array}{c|c}
[f_{i-j}]_{i,j=1}^n & [f_{i-j-n}]_{i,j=1}^n\vphantom{\Big|}\\
\hline
[f_{n+i-j}]_{i,j=1}^n & [f_{i-j}]_{i,j=1}^{n}\vphantom{\Big|}
\end{array}\right]\notag\\
&=\left[\begin{array}{c|c}
T_n(f) & J_n(f)\vphantom{\Big|}\\
\hline
K_n(f) & T_n(f)\vphantom{\Big|}
\end{array}\right],\label{T2n-Tn}
\end{align}
where
\begin{align}
J_n(f)&=[f_{i-j-n}]_{i,j=1}^n=\left[\begin{array}{ccccc}
f_{-n} & \cdots & \cdots & \cdots & f_{-(2n-1)}\\
\vdots & \ddots & & & \vdots\\
f_{-3} & & \ddots & & f_{-(n+2)}\\
f_{-2} & f_{-3} & & \ddots & f_{-(n+1)}\\[5pt]
f_{-1} & f_{-2} & f_{-3} & \cdots & f_{-n}
\end{array}\right],\notag\\[3pt]
K_n(f)&=[f_{n+i-j}]_{i,j=1}^n=\left[\begin{array}{ccccc}
f_n & \cdots & f_3 & f_2 & f_1\\
f_{n+1} & \ddots & & f_3 & f_2\\
f_{n+2} & & \ddots & & f_3\\
\vdots & & & \ddots & \vdots\\
f_{2n-1} & \cdots & \cdots & \cdots & f_n
\end{array}\right]=H_n(f)(W_n\otimes I_t).\label{K=H}
\end{align}
Note that
\[ \mathop{\rm blockdiag}_{n,n}^{s,t}(T_{2n}(f))=T_n(f)\oplus T_n(f)=\left[\begin{array}{c|c}
T_n(f) & O\vphantom{\Big|}\\
\hline
O & T_n(f)\vphantom{\Big|}
\end{array}\right] \]
is the ``block diagonal part'' of $T_{2n}(f)$ as it appears in \eqref{T2n-Tn}.
By Theorem~\ref{blockdiag(GLT)=GLT} and {\bf GLT3}, we have
\[ \Bigl\{T_{2n}(f)-\mathop{\rm blockdiag}_{n,n}^{s,t}(T_{2n}(f))\Bigr\}_n\sim_{\rm GLT}O_{s,t}. \]
Hence, by {\bf GLT2},
\begin{equation*}
\Bigl\{T_{2n}(f)-\mathop{\rm blockdiag}_{n,n}^{s,t}(T_{2n}(f))\Bigr\}_n\sim_\sigma0,
\end{equation*}
i.e.,
\begin{equation}\label{final_alfa}
\left\{\left[\begin{array}{c|c}
O & J_n(f)\vphantom{\Big|}\\
\hline
K_n(f) & O\vphantom{\Big|}
\end{array}\right]\right\}_n\sim_\sigma0.
\end{equation}
It follows from \eqref{final_alfa} that
\begin{equation}\label{KJ0}
\{K_n(f)\}_n\sim_\sigma0.
\end{equation}
To see this, apply Lemma~\ref{lemma_sigma} with
\[ X_n=\left[\begin{array}{c|c}
O & J_n(f)\vphantom{\Big|}\\
\hline
K_n(f) & O\vphantom{\Big|}
\end{array}\right],\qquad
P_n=\left[\begin{array}{c}O_n\vphantom{\Big|}\\\hline I_n\vphantom{\Big|}\end{array}\right]\otimes I_s=\left[\begin{array}{c}O_{ns}\vphantom{\Big|}\\\hline I_{ns}\vphantom{\Big|}\end{array}\right],\qquad P_n'=\left[\begin{array}{c}I_n\vphantom{\Big|}\\\hline O_n\vphantom{\Big|}\end{array}\right]\otimes I_t=\left[\begin{array}{c}I_{nt}\vphantom{\Big|}\\\hline O_{nt}\vphantom{\Big|}\end{array}\right], \]
after observing that $P_n^*X_nP_n'=K_n(f)$.
The singular value distribution \eqref{KJ0} immediately implies the thesis $\{H_n(f)\}_n\sim_\sigma0$, because $H_n(f)$ coincides with $K_n(f)$ up to the permutation matrix $W_n\otimes I_t$ appearing in \eqref{K=H}, and so $H_n(f)$ and $K_n(f)$ have the same singular values.
\end{proof}

Our proof of Theorem~\ref{thm:FT} amounts to applying Theorem~\ref{blockdiag(GLT)=GLT} and Lemma~\ref{lemma_sigma}, after observing that the bottom left block $K_n(f)$ of the matrix $T_{2n}(f)-\mathop{\rm blockdiag}_{n,n}^{s,t}(T_{2n}(f))$ coincides with the Hankel matrix $H_n(f)$ multiplied by the permutation $W_n\otimes I_t$; see \eqref{K=H}.
In the case where the Toeplitz sequence $\{T_{2n}(f)\}_n$ is replaced by a GLT sequence $\{A_n\}_n\sim_{\rm GLT}\kappa$, with $A_n$ of size $2ns\times 2nt$ just like $T_{2n}(f)$, the bottom left block $K_n$ of the matrix $A_n-\mathop{\rm blockdiag}_{n,n}^{s,t}(A_n)$ obviously coincides with some matrix $H_n$ multiplied by the permutation $W_n\otimes I_t$; it suffices to take $H_n=K_n(W_n\otimes I_t)$. In this scenario, $\{H_n\}_n$ is referred to as a GLH sequence.

\begin{definition}[\textbf{generalized locally Hankel sequence}]\label{GLH_def}
Let $\{A_n\}_n\sim_{\rm GLT}\kappa$, where $\kappa:[0,1]\times[-\pi,\pi]\to\mathbb C^{s\times t}$ is measurable and $A_n$ has size $2ns\times 2nt$ for some fixed positive integers $s,t$. For every $n$, write
\[ A_n-\mathop{\rm blockdiag}_{n,n}^{s,t}(A_n)=\left[\begin{array}{c|c}
O & J_n\vphantom{\Big|}\\
\hline
K_n & O\vphantom{\Big|}
\end{array}\right] \]
and define $H_n=K_n(W_n\otimes I_t)$. Then, $\{H_n\}_n$ is by definition a generalized locally Hankel (GLH) sequence. 
In view of \eqref{K=H}, we have $H_n=H_n(f)$ in the case where $A_n=T_{2n}(f)$.
\end{definition}

Through a simple adaptation of the proof of Theorem~\ref{thm:FT}, 
one can show that any GLH sequence is zero-distributed (and hence it is a GLT sequence with zero symbol by {\bf GLT2}). We can therefore state the following result, which is an extension of Theorem~\ref{thm:FT}.

\begin{theorem}[\textbf{extended Fasino--Tilli theorem}]\label{thm:FTextended}
Let $\{H_n\}_n$ be a GLH sequence as per Definition~{\rm\ref{GLH_def}}. Then, $\{H_n\}_n\sim_\sigma0$ and $\{H_n\}_n\sim_{\rm GLT}O_{s,t}$.
\end{theorem}

\section{Numerical experiments}\label{sec:ne}

In this section, we illustrate through numerical examples the efficiency of block Jacobi and block Gauss--Seidel preconditioners for GLT sequences, as predicted by Remark~\ref{bJp-efficiency}.
In each example, we proceed as follows.
\begin{enumerate}[nolistsep,leftmargin=20pt]
	\item[(A)] We fix a GLT sequence $\{A_n\}_n\sim_{\rm GLT}\kappa$, where $A_n$ is an invertible square matrix of size $ns$ for some fixed positive integer $s$ and $\kappa:[0,1]\times[-\pi,\pi]\to\mathbb C^{s\times s}$ is a measurable function.
	\item[(B)] For every $\nu\ge2$ and every $n\ge\nu$, we define $(n_1,\ldots,n_\nu)$ as the partition of $n$ given by
	\[ n_i=\lfloor n/\nu\rfloor,\qquad i=1,\ldots,\nu-1,\qquad n_\nu=n-(\nu-1)\lfloor n/\nu\rfloor, \]
	and we consider the corresponding block Jacobi and block Gauss--Seidel preconditioners for $A_n$, denoted by $P_n^J=\mathop{\rm blockdiag}_{n_1,\ldots,n_\nu}^{s,s}(A_n)$ and $P_n^{GS}=\mathop{\rm blocktril}_{n_1,\ldots,n_\nu}^{s,s}(A_n)$, respectively. Both $P_n^J$ and $P_n^{GS}$ will turn out to be invertible just like $A_n$.
	\item[(C)] In the case where the matrices $A_n$ are HPD, the matrices $P_n^J$ are HPD as well and we report in a table, for increasing values of $\nu$ and $n$, the number of iterations needed by the preconditioned conjugate gradient (PCG) method with preconditioner $P_n^J$ for solving the linear system $A_n\xx=\mathbf1$ up to a precision of $10^{-6}$, where $\mathbf1$ is the vector of all ones. The PCG method is started with the zero vector $\mathbf0$. We also report in another table, for the same values of $\nu$ and $n$ considered in the previous table, the number of iterations needed by the preconditioned generalized minimal residual (PGMRES) method with preconditioner $P_n^{GS}$ for solving the linear system $A_n\xx=\mathbf1$ up to a precision of $10^{-6}$. The PGMRES method is started with the zero vector $\mathbf0$ and applied without restarting.
	\item[(D)] In the case where the matrices $A_n$ are not HPD, we report in a table, for increasing values of $\nu$ and $n$, the number of iterations needed by the PGMRES method with preconditioner $P_n^J$ for solving the linear system $A_n\xx=\mathbf1$ up to a precision of $10^{-6}$. We also report in another table, for the same values of $\nu$ and $n$ considered in the previous table, the number of iterations needed by the PGMRES method with preconditioner $P_n^{GS}$ for solving the linear system $A_n\xx=\mathbf1$ up to a precision of $10^{-6}$. As in item~(C), the PGMRES method is started with the zero vector $\mathbf0$ and applied without restarting.
\end{enumerate}
In all the considered examples, for each fixed $\nu$, the number of PCG/PGMRES iterations is bounded from above by a constant $C_\nu$ independent of $n$. This shows an optimal convergence rate of both methods and, consequently, the efficiency of the preconditioners $P_n^J$ and $P_n^{GS}$ for $A_n$.

\begin{example}[\textbf{full Toeplitz matrices}]\label{e1}
Let $A_n=T_n(f)$ with $f(\theta)=|\theta|$.
In this case, we have $\{A_n\}_n\sim_{\rm GLT}f$ by {\bf GLT2}. The number $s$ in item~(A) is given by $s=1$.
Since the matrices $A_n$ are HPD by well-known properties of Toeplitz matrices---see, e.g., \cite[eq.~(6.9) and Theorem~6.1]{GLTbookI}---we follow the program in item~(C). The results are collected in Tables~\ref{e1T}--\ref{e1T'}.
\begin{table}
\footnotesize
\centering
\caption{Example~\ref{e1} --- Number of PCG iterations with preconditioner $P_n^J$ for solving $A_n\xx=\mathbf1$ up to a precision of $10^{-6}$.}
\label{e1T}
\begin{tabular}{rrrrrrrrr}
\toprule
 & \multicolumn{8}{c}{PCG iterations with preconditioner $P_n^J$$_{\vphantom{\int}}$}\\
\cline{2-9}
$n$ & $\hphantom{100}\nu=2$ & $\hphantom{100}\nu=3$ & $\hphantom{100}\nu=4$ & $\hphantom{100}\nu=5$ & $\hphantom{100}\nu=6$ & $\hphantom{100}\nu=7$ & $\hphantom{100}\nu=8$ & $\hphantom{100}\nu=9^{\vphantom{\int^0}}$\\
\midrule
40 & 5 & 9 & 8 & 9 & 14 & 15 & 12 & 16\\
80 & 6 & 10 & 9 & 10 & 15 & 16 & 13 & 18\\
160 & 6 & 10 & 10 & 11 & 16 & 18 & 14 & 19\\
320 & 6 & 11 & 10 & 12 & 17 & 19 & 16 & 21\\
640 & 7 & 12 & 11 & 12 & 18 & 19 & 17 & 20\\
1280 & 7 & 12 & 12 & 13 & 19 & 20 & 18 & 22\\
2560 & 7 & 13 & 12 & 14 & 20 & 21 & 19 & 23\\
\bottomrule
\end{tabular}
\vspace{20pt}
\caption{Example~\ref{e1} --- Number of PGMRES iterations with preconditioner $P_n^{GS}$ for solving $A_n\xx=\mathbf1$ up to a precision of $10^{-6}$.}
\label{e1T'}
\begin{tabular}{rrrrrrrrr}
\toprule
 & \multicolumn{8}{c}{PGMRES iterations with  preconditioner $P_n^{GS}$$_{\vphantom{\int}}$}\\
\cline{2-9}
$n$ & $\hphantom{100}\nu=2$ & $\hphantom{100}\nu=3$ & $\hphantom{100}\nu=4$ & $\hphantom{100}\nu=5$ & $\hphantom{100}\nu=6$ & $\hphantom{100}\nu=7$ & $\hphantom{100}\nu=8$ & $\hphantom{100}\nu=9^{\vphantom{\int^0}}$\\
\midrule
40 & 4 & 6 & 7 & 8 & 9 & 9 & 10 & 10\\
80 & 4 & 6 & 7 & 9 & 10 & 10 & 11 & 11\\
160 & 4 & 6 & 8 & 9 & 11 & 11 & 12 & 13\\
320 & 5 & 7 & 8 & 10 & 11 & 12 & 13 & 13\\
640 & 5 & 7 & 9 & 10 & 12 & 12 & 13 & 14\\
1280 & 5 & 7 & 9 & 11 & 12 & 13 & 13 & 15\\
2560 & 5 & 8 & 9 & 11 & 12 & 13 & 14 & 15\\
\bottomrule
\end{tabular}
\end{table}
\end{example}

\begin{example}[\textbf{preconditioned block Toeplitz matrices}]\label{e2}
Let $A_n=T_n(h)^{-1}T_n(f)$, where
\[ f(\theta)=\frac13\begin{bmatrix}4 & -2-2\hspace{1pt}{\rm e}^{{\rm i}\theta}\\[3pt] -2-2\hspace{1pt}{\rm e}^{-{\rm i}\theta} & 8-4\cos\theta\end{bmatrix},\qquad h(\theta)=\frac1{30}\begin{bmatrix}4 & 3+3\hspace{1pt}{\rm e}^{{\rm i}\theta}\\[3pt] 3+3\hspace{1pt}{\rm e}^{-{\rm i}\theta} & 12+2\cos\theta\end{bmatrix}. \]
The choice of $f$ and $h$ is inspired by the quadratic $C^0$ B-spline Galerkin discretization of the one-dimensional Laplacian eigenvalue problem; see \cite[Section~2.3.2]{Tom-paper}.
Since $h(\theta)$ is HPD for every $\theta\in[-\pi,\pi]$, we have $\{A_n\}_n\sim_{\rm GLT}h^{-1}f$ by {\bf GLT2}--{\bf GLT3}. The number $s$ in item~(A) is given by $s=2$.
Since the matrices $A_n$ are not HPD, we follow the program in item~(D). The results are collected in Tables~\ref{e2T}--\ref{e2T'}.
\begin{table}
\footnotesize
\centering
\caption{Example~\ref{e2} --- Number of PGMRES iterations with preconditioner $P_n^J$ for solving $A_n\xx=\mathbf1$ up to a precision of $10^{-6}$.}
\label{e2T}
\begin{tabular}{rrrrr}
\toprule
 & \multicolumn{4}{c}{PGMRES iterations with preconditioner $P_n^J$$_{\vphantom{\int}}$}\\
\cline{2-5}
$n$ & $\hphantom{100}\nu=2$ & $\hphantom{100}\nu=4$ & $\hphantom{100}\nu=8$ & $\hphantom{100}\nu=16^{\vphantom{\int^0}}$\\
\midrule
40 & 4 & 10 & 22 & 40\\
80 & 4 & 10 & 22 & 43\\
160 & 4 & 10 & 22 & 43\\
320 & 4 & 10 & 22 & 44\\
640 & 4 & 8 & 19 & 44\\
1280 & 4 & 8 & 16 & 35\\
2560 & 4 & 7 & 13 & 27\\
\bottomrule
\end{tabular}
\vspace{20pt}
\centering
\caption{Example~\ref{e2} --- Number of PGMRES iterations with preconditioner $P_n^{GS}$ for solving $A_n\xx=\mathbf1$ up to a precision of $10^{-6}$.}
\label{e2T'}
\begin{tabular}{rrrrr}
\toprule
 & \multicolumn{4}{c}{PGMRES iterations with preconditioner $P_n^{GS}$$_{\vphantom{\int}}$}\\
\cline{2-5}
$n$ & $\hphantom{100}\nu=2$ & $\hphantom{100}\nu=4$ & $\hphantom{100}\nu=8$ & $\hphantom{100}\nu=16^{\vphantom{\int^0}}$\\
\midrule
40 & 2 & 4 & 8 & 17\\
80 & 2 & 4 & 9 & 17\\
160 & 2 & 4 & 9 & 17\\
320 & 2 & 4 & 9 & 18\\
640 & 2 & 5 & 9 & 18\\
1280 & 2 & 5 & 9 & 19\\
2560 & 2 & 5 & 10 & 19\\
\bottomrule
\end{tabular}
\end{table}
\end{example}

\begin{example}[\textbf{finite element matrices}]\label{e3}
Let $a:[0,1]\to\mathbb R$ be a function in $L^1([0,1])$ with $a>0$ a.e., let $\varphi_1,\ldots,\varphi_n:[0,1]\to\mathbb R$ be the so-called ``hat-functions'' defined on the uniform grid $x_i=\frac i{n+1^{\vphantom{\int}}}$, $i=0,\ldots,n+1$, i.e.,
\[ \varphi_i(x)=\frac{x-x_{i-1}}{x_i-x_{i-1}}\hspace{1pt}\chi_{[x_{i-1},x_i)}(x)+\frac{x_{i+1}-x}{x_{i+1}-x_i}\hspace{1pt}\chi_{[x_i,x_{i+1})}(x),\qquad i=1,\ldots,n, \]
let
\begin{align*}
K_n&=\left[\int_0^1a(x)\varphi_j'(x)\varphi_i'(x){\rm d}x\right]_{i,j=1}^n,\\
H_n&=\left[\int_0^1\varphi_j'(x)\varphi_i(x){\rm d}x\right]_{i,j=1}^n=-{\rm i}\,T_n(\sin\theta),\\
M_n&=\left[\int_0^1\varphi_j(x)\varphi_i(x){\rm d}x\right]_{i,j=1}^n=\frac{1}{3(n+1)}\,T_n(2+\cos\theta),
\end{align*}
and let
\[ A_n = (n+1)(M_n + H_n^*K_n^{-1}H_n). \]
The matrix $A_n$ arises from the finite element discretization of a system of differential equations, it is HPD for all $n$ (due to the assumption that $a>0$ a.e.), and we have
\[ \{A_n\}_n\sim_{\rm GLT}\kappa(x,\theta)=\frac{2+\cos\theta}3+\frac{1+\cos\theta}{2\hspace{1pt}a(x)}; \]
see \cite[Section~10.6.2]{GLTbookI} for details. The number $s$ in item~(A) is given by $s=1$.
Since the matrices $A_n$ are HPD, we follow the program in item~(C). The results are collected in Tables~\ref{e3T}--\ref{e3T'}.
\begin{table}
\footnotesize
\centering
\caption{Example~\ref{e3} --- Number of PCG iterations with preconditioner $P_n^J$ for solving $A_n\xx=\mathbf1$ up to a precision of $10^{-6}$ in the case where $a(x)=1+\sqrt x$.}
\label{e3T}
\begin{tabular}{rrrrr}
\toprule
 & \multicolumn{4}{c}{PCG iterations with preconditioner $P_n^J$$_{\vphantom{\int}}$}\\
\cline{2-5}
$n$ & $\hphantom{100}\nu=2$ & $\hphantom{100}\nu=7$ & $\hphantom{100}\nu=12$ & $\hphantom{100}\nu=17^{\vphantom{\int^0}}$\\
\midrule
40 & 5 & 7 & 8 & 9\\
80 & 4 & 7 & 7 & 8\\
160 & 5 & 7 & 6 & 6\\
320 & 5 & 7 & 6 & 6\\
640 & 5 & 7 & 6 & 6\\
1280 & 5 & 7 & 6 & 6\\
2560 & 5 & 7 & 6 & 6\\
\bottomrule
\end{tabular}
\vspace{20pt}
\centering
\caption{Example~\ref{e3} --- Number of PGMRES iterations with preconditioner $P_n^{GS}$ for solving $A_n\xx=\mathbf1$ up to a precision of $10^{-6}$ in the case where $a(x)=1+\sqrt x$.}
\label{e3T'}
\begin{tabular}{rrrrr}
\toprule
 & \multicolumn{4}{c}{PGMRES iterations with preconditioner $P_n^{GS}$$_{\vphantom{\int}}$}\\
\cline{2-5}
$n$ & $\hphantom{100}\nu=2$ & $\hphantom{100}\nu=7$ & $\hphantom{100}\nu=12$ & $\hphantom{100}\nu=17^{\vphantom{\int^0}}$\\
\midrule
40 & 3 & 6 & 6 & 7\\
80 & 3 & 6 & 6 & 6\\
160 & 3 & 6 & 6 & 6\\
320 & 3 & 6 & 6 & 6\\
640 & 3 & 6 & 6 & 6\\
1280 & 3 & 5 & 6 & 6\\
2560 & 3 & 5 & 6 & 6\\
\bottomrule
\end{tabular}
\end{table}
\end{example}

\section{Conclusions and perspectives}\label{sec:conc}
In our main results (Theorems~\ref{blockdiag(GLT)=GLT}--\ref{blocktril(GLT)=GLT}), we have proved that any sequence of block Jacobi or block Gauss--Seidel preconditioners $\{P_n\}_n$ associated with a GLT sequence $\{A_n\}_n\sim_{\rm GLT}\kappa$ is again a GLT sequence with the same symbol $\kappa$, as long as the number of blocks appearing in each preconditioner $P_n$ is a fixed number $\nu$ independent of~$n$.
As a consequence of this result:
\begin{enumerate}[nolistsep,leftmargin=18.5pt]
	\item[(a)] in Remark~\ref{bJp-efficiency}, we have predicted the efficiency of block Jacobi and block Gauss--Seidel preconditioners for GLT sequences; an efficiency that has been illustrated through numerical experiments in Section~\ref{sec:ne};
	\item[(b)] in Theorem~\ref{thm:FTextended}, we have extended the Fasino--Tilli theorem on the zero distribution of Hankel sequences generated by $L^1$ functions to the larger class of GLH sequences defined in Definition~\ref{GLH_def}.
\end{enumerate}
We conclude this paper by suggesting some possible future lines of research.
\begin{itemize}[nolistsep,leftmargin=*]
	\item In Theorems~\ref{blockdiag(GLT)=GLT}--\ref{blocktril(GLT)=GLT}, the considered block Jacobi and block Gauss--Seidel preconditioners have a fixed number $\nu$ of blocks independent of~$n$. From a theoretical point of view, it could be interesting to investigate whether Theorems~\ref{blockdiag(GLT)=GLT}--\ref{blocktril(GLT)=GLT} continue to hold if $\nu=\nu_n$ depends on $n$ and satisfies some additional assumptions, such as $\nu=\nu_n=o(d_n)$ as $n\to\infty$.
	
	\item The block Jacobi preconditioner $P_n^J$ considered in Theorem~\ref{blockdiag(GLT)=GLT} for a matrix $A_n$ belonging to a GLT sequence $\{A_n\}_n\sim_{\rm GLT}\kappa$ is made up of blocks $A_n^{(1)},\ldots,A_n^{(\nu)}$ that are diagonal submatrices of $A_n$. In general, these submatrices are not computationally easy, i.e., solving a linear system associated with them can be computationally expensive. A possible way to improve the efficiency of the preconditioner $P_n^J$ is to replace $A_n^{(1)},\ldots,A_n^{(\nu)}$ with computationally easier ``GLT blocks'' $\tilde A_n^{(1)},\ldots,\tilde A_n^{(\nu)}$ such that the new preconditioner $\tilde P_n^J=\tilde A_n^{(1)}\oplus\cdots\oplus\tilde A_n^{(\nu)}$ satisfies $\{\tilde P_n^J\}_n\sim_{\rm GLT}\tilde\kappa$ with either $\tilde\kappa\approx\kappa$ or, even better, $\tilde\kappa=\kappa$.
	To achieve the desired goal, the ``GLT blocks'' $\tilde A_n^{(1)},\ldots,\tilde A_n^{(\nu)}$ could be defined as appropriate products between diagonal sampling matrices and computationally easy Toeplitz-like matrices such as circulant matrices \cite{Davis} or $\tau$~matrices \cite{BC}.
	
	\item In view of possible multidimensional applications of block Jacobi/Gauss--Seidel preconditioning to GLT sequences, an interesting topic for future research is to extend Theorems~\ref{blockdiag(GLT)=GLT}--\ref{blocktril(GLT)=GLT} to multilevel GLT sequences, which arise, for instance, in the numerical discretization of partial differential equations \cite{rg,bgR,bgd,GLTbookII}.
	
	\item In the numerical discretization of differential problems by domain decomposition methods (DDMs), two popular iterative solution methods are the additive and multiplicative Schwarz methods; see \cite{DDM-2024,DDM-book}.
	What is relevant to us is that, roughly speaking, the additive Schwarz method is another name for the block Jacobi method and the multiplicative Schwarz method is another name for the block Gauss--Seidel method.
	Thus, DDMs are a research field where the results of the present paper, along with their multidimensional extensions mentioned in the previous item, could find application.
	More specifically, Theorems~\ref{blockdiag(GLT)=GLT}--\ref{blocktril(GLT)=GLT} and their multidimensional extensions could be used in combination with the basic theoretical tools provided in \cite[Section~3.2]{GLT-vs-Fourier} to analyze the additive and multiplicative Schwarz methods applied to GLT structures arising from DDM discretizations. A first contribution in this direction has recently appeared in \cite{Rifqui-DDM-1d}.
	
	\item A successful idea proposed by Pestana and Wathen \cite{PW} for solving a real non-symmetric Toeplitz linear system $T_n(f)\xx=\boldsymbol b$ consists in pre-multiplying both sides by the flip matrix \eqref{flipm} and solving the resulting real symmetric linear system $W_nT_n(f)\xx=W_n\boldsymbol b$ by either the minimal residual (MINRES) method or its preconditioned version.
	This idea, which we call ``flipping strategy'', has gained more and more popularity over time and it has now become a consolidated research area.
	What is relevant to us is that the spectral distribution of the sequence of system matrices $W_nT_n(f)$, which plays a significant role in the convergence analysis of MINRES, was established in \cite{SIMAX-hankel,BIT-hankel} by leveraging on the Fasino--Tilli theorem; see also the generalization to the multilevel case performed in \cite{ELA-multihankel,SIMAX-multihankel}, which again invokes (the multilevel version of) the Fasino--Tilli theorem.
	Moreover, in the spectral distribution analysis of block matrices with rectangular Toeplitz blocks---which arise in several applications as explained in \cite[Section~1]{blocking}---the Fasino--Tilli theorem is again a cornerstone; see \cite{blocking} for more details.
	Given the central role played by the Fasino--Tilli theorem in the asymptotic spectral analysis of Toeplitz-like matrices such as those mentioned above, we can predict that the same central role is played by the extended Fasino--Tilli theorem (Theorem~\ref{thm:FTextended}) in the asymptotic spectral analysis of ``GLT-like'' matrices. This observation opens the way to several future investigations on the possible applications of the extended Fasino--Tilli theorem, including in particular the spectral distribution analysis of: (a) flipped ``GLT matrices'' resulting from the application of the flipping strategy to ``GLT linear systems'' $A_n\xx=\boldsymbol b$ with $\{A_n\}_n\sim_{\rm GLT}\kappa$; (b) block matrices with rectangular ``GLT blocks''.
\end{itemize}

\section*{Acknowledgements}

{\footnotesize

The authors are members of the research group GNCS (Gruppo Nazionale per il Calcolo Scientifico) of INdAM (Istituto Nazionale di Alta Matematica).
This work was supported
by the INdAM-GNCS project PASTRAMI (sPline And Solver innovaTions foR Adaptive isogeoMetric analysIs, CUP E53C24001950001),
by the MUR excellence department project MatMod@TOV awarded to the Department of Mathematics of the University of Rome Tor Vergata (CUP E83C23000330006),
by the Department of Mathematics of the University of Rome Tor Vergata through the project METRO (Methods and modEls for arTificial neuRal netwOrks, CUP E83C25000630005),
by the Theory, Economics and Systems Laboratory (TESLAB) of the Department of Computer Science of the Athens University of Economics and Business,
by the PRIN-PNRR project MATHPROCULT (MATHematical tools for predictive maintenance and PROtection of CULTural heritage, Code P20228HZWR, CUP J53D23003780006),
and by the European High-Performance Computing Joint Undertaking (JU) under grant agreement No 955701. The JU receives support from the European Union's Horizon 2020 research and innovation programme and Belgium, France, Germany, Switzerland.}

\end{document}